\documentclass[11pt]{article}
\usepackage{amsmath,amsthm,amssymb,amsxtra, graphicx}
\usepackage{txfonts}
\usepackage{color}

\usepackage[
colorlinks=true,
linkcolor=blue,
 citecolor=blue,
  urlcolor=blue,
]{hyperref}

\date{}

\newcommand{\bdop}{\mathcal{B}}

\newtheorem{theorem}{Theorem}[section]
\newtheorem{lemma}[theorem]{Lemma}
\newtheorem{prop}[theorem]{Proposition}
\newtheorem{cor}[theorem]{Corollary}
\newtheorem{conjecture}[theorem]{Conjecture}

\theoremstyle{definition}

\newtheorem{definition}[theorem]{Definition}

\newcommand{\strt}[1]{\rule{0pt}{#1}}

\newcommand{\absval}[1]{\mbox{$|#1|$}}
\newcommand{\ud}[0]{\,\mathrm{d}}

\theoremstyle{definition}

\usepackage[iso]{isodateo}
\usepackage{enumerate}
\renewcommand{\thefootnote}{\fnsymbol{footnote}}

\newcommand{\re}{\mathbb{R}}
\newcommand{\R }{\mathbb R}
\newcommand{\Z}{\mathbb{Z}}

\newcommand{\cd}{\mathcal D}
\newcommand{\cs}{\mathcal S}

\let \la=\lambda
\let \e=\varepsilon
\let \d=\delta

\let \g=\gamma
\let \Om=\Omega

\let \si=\sigma

\newcommand{\abs}[1]{|#1|}

\newcommand{\norm}[1]{\mbox{$\left\| #1 \right\|$}}
\newcommand{\Norm}[2]{\|#1\|_{#2}}

\def\Xint#1{\mathchoice
  {\XXint\displaystyle\textstyle{#1}}%
  {\XXint\textstyle\scriptstyle{#1}}%
  {\XXint\scriptstyle\scriptscriptstyle{#1}}%
  {\XXint\scriptscriptstyle\scriptscriptstyle{#1}}%
  \!\int}
\def\XXint#1#2#3{{\setbox0=\hbox{$#1{#2#3}{\int}$}
    \vcenter{\hbox{$#2#3$}}\kern-.5\wd0}}

\def\avgint{\Xint-}

\begin{document}
\title{ \bf  The $L(\log L)^{\epsilon}$ endpoint estimate for maximal singular integral operators}

\author{Tuomas Hyt\"onen  \,and\, Carlos P\'erez}
\renewcommand{\thefootnote}{\fnsymbol{footnote}}
 \footnotetext{The
first author is supported by the European Union through the ERC Starting Grant ``Analytic--probabilistic methods for borderline singular integrals''. He is part of the Finnish Centre of Excellence in Analysis and Dynamics Research.}
\footnotetext{The second author was
supported by the Spanish Ministry of Science and Innovation grant
MTM2012-30748.}

\maketitle

\begin{abstract}

We prove in this paper the following estimate for the maximal operator $T^*$ associated to the singular integral operator $T$:
\begin{equation*}
\|T^*f\|_{L^{1,\infty}(w)}  \lesssim  \frac{1}{\epsilon}
 \int_{\R^n} |f(x)|\,M_{\strt{1.7ex} L(\log L)^{\epsilon}} (w)(x)\,dx,   \qquad w\geq 0, \, \, 0<\epsilon \leq 1.
\end{equation*}
This follows from the sharp $L^p$ estimate
\begin{equation*}
\|T^*f \|_{ L^{p}(w) }\,   \lesssim     p'\,  (\frac{1}{\delta})^{1/p'}\,  \|f \|_{\strt{1.7ex} L^{p}(
M_{ L(\log L)^{p-1+\d}} (w))},  \qquad 1<p<\infty, w\geq 0, \,\,0<\d \leq1.
\end{equation*}
As as a consequence we deduce that
\begin{equation*}
\|T^*f\|_{L^{1,\infty}(w)} \lesssim [w]_{A_{1}}   \log(e+ [w]_{A_{\infty}})
\, \int_{\R^n} |f|\, w\,dx,
\end{equation*}
extending the endpoint results obtained in \cite{LOP} and \cite{HP} to maximal singular integrals. 
Another consequence is a quantitative two weight bump estimate.

\end{abstract}

\section{Introduction and main results}

Very recently, the so called Muckenhoupt-Wheeden conjecture has been disproved by Reguera-Thiele in \cite{RT}. This conjecture claimed that there exists a constant $c$ such that for any function $f$ and any weight $w$ (i.e., a nonnegative locally integrable function), there holds
\begin{equation}\label{conjecture-mw}
\|Hf\|_{L^{1,\infty}(w)}\leq c\,\int_{\R} \,|f|\,Mwdx.
\end{equation}
where $H$ is the Hilbert transform. The failure of the conjecture was previously obtained by M.C. Reguera in \cite{R} for a special model operator $T$  instead of $H$. This conjecture was motivated by a similar inequality by C. Fefferman and E. Stein \cite{FS} for the 
Hardy-Littlewood maximal function: 
\begin{equation}\label{fs}
\|Mf\|_{L^{1,\infty}(w)}\leq c\, \int_{\R^n}
|f|\,Mw\,dx.
\end{equation}
The importance of this result stems from the fact that it was a central piece in the approach by Fefferman-Stein to derive the
following vector-valued extension of the classical $L^p$ Hardy-Littlewood
maximal theorem: for every $1<p,q<\infty$, there is a finite
constant $c=c_{p,q}$ such that
\begin{equation}\label{fs3}
\bigg\|\Big( \sum_j (Mf_j)^q \Big)^\frac1q\bigg\|_{L^p(\R^n)} \le
c\, \bigg\|\Big( \sum_j |f_j|^q \Big)^\frac1q\bigg\|_{L^p(\R^n)}.
\end{equation}
This is a very deep theorem and has been used a lot in modern
harmonic analysis explaining the central role of inequality \eqref{fs}.

Inequality \eqref{conjecture-mw} was conjectured by B. Muckenhoupt and  R. Wheeden during the 70's. That this conjecture was believed to be false was already mentioned in \cite{P2}  where the best positive result in this direction so far can be found, and where $M$ is replaced by $M_{L(\log L)^{\epsilon}}$, i.e., a maximal type operator that is ``$\epsilon$-logarithmically'' bigger than $M$: 
\begin{equation*}
\|Tf\|_{L^{1,\infty}(w)}\leq c_{\e} \,\int_{{\mathbb
R}^n}|f|\,M_{ L(\log
L)^{\e} }(w)dx \qquad w\geq 0.
\end{equation*}
where $T$ is the Calder\'on-Zygmund operator $T$. Until very recently the constant of the estimate did not play any essential role except, perhaps, for the fact that it blows up. If we check the computations in \cite{P2} we find that $c_{\e} \approx e^{\frac{1}{\e}}$.  It turns out that improving this constant would lead to understanding deep questions in the area.  One of the main purposes  of this paper is to improve this result in several ways. A first main direction is to improve the exponential blow up $e^{\frac{1}{\e}}$ by a linear blow up $\frac{1}{\e}$. The second improvement consists of replacing $T$ by the maximal singular integral operator $T^{*}$. 
The method in \cite{P2} cannot be used directly since the linearity of $T$ played a crucial role.

We refer to Section \ref{orliczmaximal} for the definition of the maximal function $M_A=M_{A(L)}$. We remark that the operator $M_{ L(\log L)^{\e} }$ is pointwise  smaller than $M_r=M_{L^{r}}$, $r>1$, which is an $A_1$ weight and for which the result was known.

\begin{theorem} \label{EndPointResult}  
Let \, $T$\, be a Calder\'on-Zygmund operator with maximal singular integral operator $T^{*}$. Then  for any $0<\epsilon \leq 1$,
\begin{equation} \label{goalT*}
\|T^*f\|_{L^{1,\infty}(w)}  \lesssim  \frac{c_{T}}{\epsilon}
 \int_{\R^n} |f(x)|\,M_{L(\log L)^{\epsilon}} (w)(x)\,dx   \qquad w\geq 0
\end{equation}
\end{theorem}

If we formally optimize this inequality in $\epsilon$  we derive to the following conjecture:
\begin{equation} \label{conjecture}
\|T^*f\|_{L^{1,\infty}(w)}  \leq   c_{T}
 \int_{\R^n} |f(x)|\,M_{\strt{1.7ex} L\log \log L} (w)(x)\,dx    \qquad w\geq 0, \,\, f \in L_c^{\infty}(\R^n).
\end{equation}

To prove Theorem \ref{EndPointResult}  we need first an $L^p$  version of this result, which is fully sharp, at least in the logarithmic case. The result 
will hold for all \,$p\in(1,\infty)$\, but for proving Theorem \ref{EndPointResult}   we only need it when \,$p$\, is close to one.

There are two relevant properties properties that will be used (see Lemma \ref{Gene.Coifman-Rochberg}). The first one establishes that 
for appropriate $A$ and all $\g \in (0,1)$, we have $(M_{A}f)^{\g} \in A_1$\,\, with constant $[(M_{A}f)^{\g} ]_{A_1}$ independent of $A$ and $f$.
The second  property is that \,$M_{\bar{A}}$ \, is a bounded  operator on  \,$L^{p'}(\R^n)$\,  where \,$\bar{A}$ is the complementary Young function of $A$.
The main example is \, $A(t)=t^{p}(1+\log^{+} t)^{p-1+\delta},$\, $p\in(1,\infty),\, \d\in (0,\infty)$\,   since 
$$\|M_{\bar{A}}\|_{\bdop(L^{p'}(\R^n))} \lesssim p^2\, (\frac{1}{\delta})^{1/p'}
$$
by \eqref{suff3}.

\begin{theorem}\label{MainResult}
Let $1<p<\infty$ \,and  let  \,$A$\, be a Young function, then 
\begin{equation}\label{T*-p}
\|T^*f \|_{ L^{p}(w) }\,   \leq c_T    p'\,   \|M_{\bar{A}}\|_{\bdop(L^{p'}(\R^n))}\,  \|f \|_{\strt{1.7ex} L^{p}
\left ( M_{  A}(w^{1/p})^{p} \right)  }  \qquad w\geq 0. 
\end{equation}
In the particular case $A(t)=t^{p}(1+\log^{+} t)^{p-1+\delta}$ we have
\begin{equation*}
\|T^*f \|_{ L^{p}(w) }\,   \leq c_T\,   p'\,  (\frac{1}{\delta})^{1/p'}\,  
\|f \|_{\strt{1.7ex} L^{p}\left(M_{ L(\log L)^{p-1+\d}} (w)  \right)}  \qquad w\geq 0,  \quad 0<\d \leq 1. 
\end{equation*}
\end{theorem}

Another worthwhile example is given by $M_{ L   (\log L)^{p-1}       (\log \log L)^{p-1+\d}}$  instead of $M_{ L(\log L)^{p-1+\d}} $ for which:
$$\|T^*f \|_{ L^{p}(w) }\,   \leq c_T\,   p'\,  (\frac{1}{\delta})^{1/p'}\,  
\|f \|_{\strt{1.7ex} L^{p}\left( M_{ L   (\log L)^{p-1}       (\log \log L)^{p-1+\d}} (w)  \right)}  \qquad w\geq 0,  \quad 0<\d\leq 1. 
$$

There are some interesting consequences from Theorem  \ref{EndPointResult}, the first one is related to the one weight theory.    
We first recall that the definition of the $A_\infty$ constant considered in \cite{HP} and where is  shown it is the most suitable one. 
This definition was originally introduced  by Fujii in \cite{Fujii1} and rediscovered later by Wilson in \cite{W1}.

\begin{definition}
 \begin{equation*}
[w]_{A_\infty}:= \sup_Q\frac{1}{w(Q)}\int_Q M(w\chi_Q )\ dx.
 \end{equation*}

\end{definition}

Observe that $[w]_{A_\infty} \ge 1$ by the Lebesgue differentiation theorem.

When specialized to weights $w\in A_\infty$ or $w\in A_1$, Theorem~\ref{EndPointResult} yields the following corollary. It was formerly known for the linear singular integral $T$ \cite{HP}, and this was used in the proof, which proceeded via the adjoint of $T$; the novelty in the corollary below consists of dealing with the maximal singular integral $T^*$.

\begin{cor} \label{endpointA_{1}}

\begin{equation}  \label{endpoint1}
\|T^*f\|_{L^{1,\infty}(w)} \lesssim\,   \log(e+ [w]_{A_{\infty}})
\, \int_{\R^n} |f|\, Mw\,dx,
\end{equation}
and hence
\begin{equation}
\|T^*f\|_{L^{1,\infty}(w)} \lesssim [w]_{A_{1}}   \log(e+ [w]_{A_{\infty}})
\, \int_{\R^n} |f|\, w\,dx,
\end{equation}

\end{cor}

The key result that we need is the following optimal reverse H\"older's inequality obtained in \cite{HP} (see also \cite{HPR} for a better proof and \cite{DMRO} for new characterizations of the $A_{\infty}$ class of weights).

\begin{theorem}  \label{thm:SharpRHI}

Let \,$w \in A_{\infty}$, then there exists a dimensional constant $\tau_{n}$ such that
\begin{equation*}
  \Big(\avgint_Q w^{r_{w}}\Big)^{1/r_{w}}\leq 2\avgint_Q w
\end{equation*}
where 
$$r_w=1+\frac{1}{\tau_{n} [w]_{A_{\infty} }}$$

\end{theorem}

\begin{proof}[Proof of Corollary~\ref{endpointA_{1}}]
To apply \eqref{goalT*},  we use  $\log t \leq \frac{t^{\alpha}}{\alpha}$ for  $t>1$ and $\alpha>0$  to deduce that
$$M_{L(\log L)^{\epsilon}} (w) \lesssim \frac{1}{\alpha^{\epsilon}} M_{L^{1+\epsilon\alpha}}(w) $$
Hence,  if \,$w\in A_{\infty}$\, we can choose \,$\alpha$\, such that \,$\alpha\epsilon= \frac{1}{\tau_n [w]_{A_{\infty} } }$. Then, applying Theorem  \ref{thm:SharpRHI}
$$\frac{1}{\epsilon} M_{L(\log L)^{\epsilon}} (w) \lesssim \frac{1}{\epsilon}  (\epsilon \tau [w]_{A_{\infty} }  )^{\epsilon} M_{L^{r_w}}(w) 
\lesssim \frac{1}{\epsilon}  [w]_{A_{\infty} }^{\epsilon} M(w)
$$
and optimizing with \,$\epsilon \approx 1/ \log(e+ [w]_{A_{\infty}} ) $\, we obtain \eqref{endpoint1}. 
\end{proof}

As a consequence of Theorem~\ref{EndPointResult} we have, by using some variations  of   the ideas from \cite{CP1}, the following:

\begin{cor} \label{Lloglleftbumpcondit} Let $u,\si$ be a pair of weights and let  $p\in(1,\infty)$. We also let \,$\delta, \delta_{1}, \delta_{2}  \in (0,1]$. Then 
\begin{itemize}
  \item[\rm(a)] If 
\begin{equation}
K=\sup_Q\, \|u^{1/p}\|_{ \strt{1.7ex} L^p(\log L)^{p-1+\delta},Q}
\left(\frac{1}{|Q|}\int_Q \si \,dx\right)^{1/p'} < \infty,
\end{equation}
then 
\begin{equation}\label{orlicz-bump}
\|T^*(f\si)\|_{L^{p,\infty}(u)}  \lesssim  \frac{1}{\delta}\,K\, (\frac{1}{\delta})^{1/p'}   \|f\|_{L^{p}(\si)}    
\end{equation}
(The boundedness in the case $\delta=0$ is false as shown in \cite{CP1}.)  

\item[\rm(b)]  As consequence, if
\begin{equation}
\begin{split}
K &=\sup_Q\, \|u^{1/p}\|_{ \strt{1.7ex} L^p(\log L)^{p-1+\delta_{1}},Q} \,\left(\frac{1}{|Q|}\int_Q \si \,dx\right)^{1/p'} \\
&\qquad+ \sup_Q\,
\left( \frac{1}{|Q|}\int_Q u \,dx \right)^{1/p}   \, \| \sigma^{1/p'} \|_{ \strt{1.7ex} L^{p'}(\log L)^{p'-1+\delta_{2}},Q }
 < \infty,
\end{split}
\end{equation}
then 
\begin{equation}\label{Lp-bumpEstimate}
\|T^*(f\si)\|_{L^{p}(u)}  \lesssim  K\, \left(\frac{1}{\delta_{1}}\left(\frac{1}{\delta_{1}}\right)^{\frac{1}{p'}} +\frac{1}{\delta_{2}}\left(\frac{1}{\delta_{2}} \right)^{\frac{1}{p}}   \right)   \|f\|_{L^{p}(\si)}.    
\end{equation}

\end{itemize}

\end{cor}

The first qualitative result as in \eqref{orlicz-bump} was obtained in \cite{CP1}, Theorem 1.2 and its extension Theorem 4.1.

We remark that this result holds for any operator $T$ which satisfies estimate \eqref{goalT*}.  We also remark that 
this corollary improves the main results from \cite{CRV} (see also \cite{ACM}) by providing very precise quantitative estimates. We refer to these papers for historical information about this problem. 

We don't know whether the factors $\frac{1}{\delta_{i}}$, $i=1,2$ can be removed or improved from the estimate \eqref{Lp-bumpEstimate}. Perhaps our method is not so precise to prove the conjecture formulated in Section \ref{conjetura}. However, it is clear from our arguments that these factors are due to the appearance of the factor $\frac{1}{\epsilon}$ in \eqref{goalT*}.

\subsection*{Acknowledgments}

We would like to thank the anonymous referee for detailed comments that improved the presentation.

\section{Basic definitions and notation}\label{basics}

\subsection{ Singular integrals}

In this section we collect some notation and recall some classical results.  

By a Calder\'on-Zygmund operator we mean a continuous linear
operator \\  $T:C_0^{\infty}({\mathbb R}^n)\to\mathcal{D}'({\mathbb
R}^n)$ that extends to a bounded operator on $L^2({\mathbb R}^n)$,
and whose distributional kernel $K$ coincides away from the diagonal
$x=y$ in ${\mathbb R}^n\times {\mathbb R}^n$ with a function $K$
satisfying the size estimate
$$
|K(x,y)|\le \frac{c}{|x-y|^n}
$$
and the regularity condition: for some $\e>0$,
$$
|K(x,y)-K(z,y)|+|K(y,x)-K(y,z)|\le c\frac{|x-z|^{\e}}{|x-y|^{n+\e}},
$$
whenever $2|x-z|<|x-y|$, and so that
$$Tf(x)=\int_{{\mathbb R}^n}K(x,y)f(y)dy,$$
whenever $f\in C_0^{\infty}({\mathbb R}^n)$ and
$x\not\in\mbox{supp}(f)$.   

Also we will denote by $T^{*}$ the associated maximal singular integral:
$$
T^*f(x)=\sup_{\e>0}  \left| \int_{|y-x|>\e}    K(x,y) f(y)\, dy  \right| \qquad f\in C^{\infty}_{0}(\R^{n})
$$
More information can be found in many places as for instance in \cite{GrMF} or \cite{Duo}. 

\subsection{Orlicz spaces and normalized measures}\label{orlicz}

We will also need some basic facts from the theory of Orlicz spaces that we state without proof.  We refer to the book of Rao and Ren  \cite{RR} for the proofs and more information on Orlicz spaces. Another interesting recent book is \cite{W2}.

A Young function is a convex, increasing function
$A:[0,\infty)\to[0,\infty)$ with $A(0)=0$, such that 
$A(t)\to \infty$ as $t\rightarrow\infty$. Such a function is automatically continuous. From these properties it follows that $A:[t_0,\infty)\to[0,\infty)$ is a strictly increasing bijection, where $t_0=\sup\{t\in[0,\infty):A(t)=0\}$. Thus $A^{-1}(t)$ is well-defined (single-valued) for $t>0$, but in general it may happen that $A^{-1}(0)=[0,t_0]$ is an interval.

The properties of
$A$ easily imply that for $0<\e<1$ and $ t\geq 0$
\begin{equation}\label{property1}
A(\e\, t) \leq \e\,A(t)\, .
\end{equation}
The $A$-norm of a function $f$ over a set $E$ with finite measure
is defined by
$$ \|f\|_{A, E}=\|f\|_{A(L), E}=
\inf \{\la>0\,:\,  \avgint_E A\left (\frac{|f(x)|}{\la
}\right )dx\leq 1\} $$
where as usual we define the average of $f$ over a cube $E$,\, $\avgint_E f = \frac{1}{|E|}\int_E f\,dx$.

In many situations the convexity does not play any role and basically the monotonicity is the fundamental property. The convexity is used for proving that $\|\ \|_{A,E}$ is a norm which is often not required.

We will use the fact that
\begin{equation}\label{property2}
\|f\|_{A, E} \leq 1 \quad \mbox{if and only if} \quad   \avgint_E   A \left(|f(x)|\right)dx \leq 1.
\end{equation}

Associated with each Young function $A$,  one can define a
complementary function
\begin{equation}\label{complementaria}
\bar A(s)=\sup_{t>0}\{st-A(t)\} \qquad s\geq 0.
\end{equation}
Then $\bar A$ is finite-valued if and only if $\lim_{t\to\infty}A(t)/t=\sup_{t>0}A(t)/t=\infty$, which we henceforth assume; otherwise, $\bar A(s)=\infty$ for all $s>\sup_{t>0}A(t)/t$. Also, $\bar A$ is strictly increasing on $[0,\infty)$ if and only if $\lim_{t\to 0}A(t)/t=\inf_{t>0}A(t)/t=0$; otherwise $\bar A(s)=0$ for all $s\leq \inf_{t>0}A(t)/t$.

Such $\bar A$ is also a Young function and has the
property that
\begin{equation}\label{preHolder-Orlicz}
st \le A(t) +\bar A(s), \qquad t,s\geq 0.
\end{equation}
and also
\begin{equation} \label{propiedad}
t\le A^{-1}(t)\bar{A}^{-1}(t) \le 2\,t, \qquad t> 0.
\end{equation}
The main property is the following  generalized H\"older's inequality 
\begin{equation}
\frac{1}{|E|}\int_{E}|fg| dx \leq 2\|f\|_{A,E} \|g\|_{\bar{A},E}.
\label{GHI}
\end{equation}

As we already mentioned, the following Young functions play a main role in the theory:
$$A(t)=t^{p}(1+\log^{+} t)^{p-1+\delta}\qquad t,\, \d>0, p>1.
$$

\subsection{General maximal functions and $L^p$ boundedness: precise versions of old results} \label{orliczmaximal}

Given a Young function $A$ or more generally any positive function $A(t)$ we define the following maximal operator (\cite{P1},\cite{P2}) 
$$ M_{A(L)} f(x) = M_A f(x) = \sup_{Q\ni x}\|f\|_{A, Q}.$$
This operator satisfies the following distributional type estimate: there are finite dimensional constants\, $c_n,d_n$ such that 
\begin{equation}
|\{ x\in\R^{n}: M_{\strt{1.5ex} A} f(x)>t \}|  \leq
 c_{n}\, \int_{ \R^{n} }
A\left( d_{n}\frac{  f } {t} \right) \,
dx  \quad f\geq 0, \, t>0  
\label{covering}
\end{equation}
This follows from standard methods and we refer to \cite[Remark
A.3]{CMP} for details. 

A first consequence of this estimate is the following $L^p$ estimate of the operator, which is nothing more than a more precise version of one the main results from \cite{P1}. A second application will be used in the proof of Lemma \ref{Gene.Coifman-Rochberg}. 

\begin{lemma} \label{LemmaBp}
Let $A$ be a Young function, then
\begin{equation}
\|M_{A}\|_{\bdop(L^{p}(\R^n))}\,   \leq c_{n}\, \alpha_{p}(A)
\label{suff}
\end{equation}
where $\alpha_{p}(A)$ is the following tail condition that plays a central role in the sequel  
\begin{equation}
\alpha_{p}(A)= \left(\,\int_{1}^{\infty} \frac{ A(t)} { t^p }
\frac{dt}{t}\right)^{1/p}
 < \infty.
\label{Bp}
\end{equation}
\end{lemma}

Examples of functions satisfying the $B_p$ condition are
 $A(t)=t^q$, $1\leq q<p$.  More interesting examples are given by
\[  A(t) = \frac{t^p}{(1+\log^+t)^{1+\delta} }
\quad
A(t) \approx t^p\log(t)^{-1}\log\log(t)^{-(1+\delta)}, \quad p>1, \delta>0. \]

Often we need to consider instead of the function $A$ in \eqref{Bp}  the
complementary $\bar{A}$.

We also record a basic estimate between a Young function and its derivative:
\begin{equation}
A(t)\leq tA'(t)
\label{propiedad2}
\end{equation}
which holds for any $t\in(0,\infty)$ such that   $A'(t)$ does exist.

There is the following useful alternative  estimate of \eqref{suff} that will be used in the sequel.
Although variants of this lemma are well known in the literature (cf. \cite{CMP}, Proposition 5.10), we would like to stress the fact that we avoid the doubling condition on the Young functions $B$ and $\bar B$, which is important in view of the quantitative applications to follow: even if our typical Young functions are actually doubling, we want to avoid the appearance of their (large) doubling constants in our estimates.

\begin{lemma} \label{suff2} Let $B$ a Young function. Then 
\begin{equation}
\|M_{B}\|_{\bdop(L^{p}(\R^n))}\,   \leq c_{n}\,   \beta_{p}(B)
\end{equation}
where 
$$
\beta_{p}(B)= 
\left( \int_{ B(1)   }^\infty\Big(\frac{t}{\bar B(t) }\Big)^p d\bar B(t)  \right)^{1/p}
$$
\end{lemma}

\begin{proof}
We first prove that for $a>0$
\begin{equation}\label{eq:auxEst}
    \int_{B^{-1}(a)}^\infty\frac{dB(t)}{t^p}\leq \int_{ \bar B^{-1}(a) }^\infty\Big(\frac{t}{\bar B(t)}\Big)^p d\bar B(t).
\end{equation}

We discretize the integrals with a sequence $a_k:=\eta^k a$, where $\eta>1$ and eventually we pass to the limit $\eta\to 1$. Then
\begin{equation*}
    \int_{B^{-1}(a)}^\infty\frac{dB(t)}{t^p}
    =\sum_{k=1}^\infty\int_{B^{-1}(a_k)}^{B^{-1}(a_{k+1})}\frac{dB(t)}{t^p}
    \leq\sum_{k=1}^\infty\frac{1}{B^{-1}(a_k)^p}\int_{B^{-1}(a_k)}^{B^{-1}(a_{k+1})}dB(t)
    =\sum_{k=1}^\infty\frac{1}{B^{-1}(a_k)^p}(a_{k+1}-a_k).
\end{equation*}
Similarly,
\begin{equation*}
\begin{split}
  \int_{\bar B^{-1}(a)}^\infty\Big(\frac{t}{\bar B(t)}\Big)^p d\bar B(t)
  &=\sum_{k=0}^\infty\int_{\bar B^{-1}(a_k)}^{\bar B^{-1}(a_{k+1})}\Big(\frac{t}{\bar B(t)}\Big)^p d\bar B(t) \\
  &\geq\sum_{k=0}^\infty\Big(\frac{\bar B^{-1}(a_{k+1})}{\bar B(\bar B^{-1}(a_{k+1}))}\Big)^p\int_{\bar B^{-1}(a_k)}^{\bar B^{-1}(a_{k+1})} d\bar B(t)
  =\sum_{k=0}^\infty\Big(\frac{\bar B^{-1}(a_{k+1})}{a_{k+1}}\Big)^p(a_{k+1}-a_k),
\end{split}
\end{equation*}
where we used the fact that $t\mapsto \bar{B}(t)/t$ is increasing, so its reciprocal is decreasing. Moreover,
\begin{equation*}
  \frac{B^{-1}(a_{k+1})}{a_{k+1}}\geq\frac{\bar B^{-1}(a_k)}{a_{k+1}}\frac{B^{-1}(a_k)}{B^{-1}(a_k)}
  \overset{\eqref{propiedad}}{\geq}\frac{a_k}{a_{k+1}}\frac{1}{B^{-1}(a_k)}=\frac{1}{\eta B^{-1}(a_k)}
\end{equation*}
and hence
\begin{equation*}
  \int_{B^{-1}(a)}^\infty\frac{dB(t)}{t^p}
  \leq\eta^p\int_{\bar B^{-1}(a)}^\infty\Big(\frac{t}{\bar B(t)}\Big)^p d\bar B(t).
\end{equation*}
Since this is valid for any $\eta>1$, we obtain \eqref{eq:auxEst}.

Now, let $t_1=\max(1,t_0)$, where $t_0=\max\{t:B(t)=0\}$. Using $B(t)dt/t\leq dB(t)$ and applying \eqref{eq:auxEst} with $a=B(t_1+\epsilon)>0$
\begin{equation*} \begin{split}
  \alpha_p(B)  &=
  \lim_{\epsilon\to 0}\Big(\int_{t_1+\epsilon}^\infty\frac{B(t)}{t^p}\frac{dt}{t}\Big)^{1/p}   \leq
  \lim_{\epsilon\to 0}\Big(\int_{B^{-1}(B(t_1+\epsilon))}^\infty\frac{dB(t)}{t^p}\Big)^{1/p} \\  
&\overset{\text{\eqref{eq:auxEst}}}{\leq}\lim_{\epsilon\to 0}\Big(\int_{\bar B^{-1}(B(t_1+\epsilon))}^\infty \Big(\frac{t}{\bar B(t)}\Big)^p d\bar B(t)\Big)^{1/p}
\leq\Big(\int_{B(1) }^\infty\Big(\frac{t}{\bar B(t)}\Big)^p 
d\bar B(t) \Big)^{1/p},
\end{split}\end{equation*}
where in the last step we used \eqref{propiedad} with $t=B(t_1+\epsilon)$ to conclude that
\begin{equation*}
  \bar{B}^{-1}(B(t_1+\epsilon))\geq\frac{B(t_1+\epsilon)}{t_1+\epsilon}\geq\frac{B(t_1)}{t_1}\geq B(1),
\end{equation*}
since $B(t)/t$ is increasing and $t_1\geq 1$.
\end{proof}

In this paper we will consider $B$ so that $\bar{B}(t)= A(t)=  t^{p}(1+\log^{+} t)^{p-1+\delta},$\, $\delta>0$.  Then, for \,$0<\delta \leq 1$
$$A'(t)\leq 2p\frac{A(t)}{t}\qquad t>1$$
and
\begin{equation*}
  \bar A(1)=\sup_{t\in(0,1)}(t-t^p)=(t-t^p)\Big|_{t=p^{-1/(p-1)}}=(p-1)p^{-p'} .
\end{equation*}
Thus, by  the lemma
\begin{equation}
\|M_{\bar{A}}\|_ {\bdop(L^{p'}(\R^n))}  
\leq c_{n}\,
\left(
\int_{ (p-1)p^{-p'}  }^{\infty} \left( \frac{ t }{ A(t)} \right)^{p'}
A'(t)\, dt \right)^{1/p'} \leq c_n p^2\,\left( \frac{1}{\delta}\right)^{1/p'}
\label{suff3}
\end{equation}

Similarly for the smaller functional: 
$$\bar{B}(t)= A(t)= t^{p}  (1+\log^{+} t)^{p-1} (1+\log^{+}(1 +\log^{+} t) )^{p-1+\delta} \qquad \delta>0.$$ 
Then, using that \,$A'(t)\leq 3\,p\,\frac{A(t)}{t}$\, $t>1$, when $0<\delta \leq 1$ and hence by the lemma
$$
\|M_{\bar{A}}\|_ {\bdop(L^{p'}(\R^n))}
\leq c_{n}\, p^2\,\left( \frac{1}{\delta}\right)^{1/p'}
$$

\subsection{The iteration lemma}

We will need the following variation of the Rubio de Francia algorithm.
\begin{lemma} \label{smallLemma}
Let \, $1<s<\infty$ \, and let \, $v$ \, be a weight. Then there
exists a nonnegative sublinear operator $R$
satisfying the following properties:\\
(a) \quad $h\leq R(h)$\\
(b) \quad $\|R(h)\|_{L^{s}(w)}\le 2\|h\|_{L^{s}(v)}$\\
(c) \quad $R(h)v^{1/s}\in A_{1}$ with
$$[R(h)v^{1/s}]_{A_{1}}\le
cs^{\prime}$$
\end{lemma}

\begin{proof}

We consider the operator
$$
S(f)=\frac{M(f\,v^{1/s})}{v^{1/s}}
$$
Since $\|M\|_{L^{s}}\sim s^{\prime}$, we have
$$
\|S(f)\|_{L^{s}(v)}\le cs^{\prime}\|f\|_{L^{s}(v)}.
$$
Now, define the Rubio de Francia  operator $R$ by
$$
R(h)=\sum_{k=0}^{\infty}\frac{1}{2^{k}}\frac{S^{k}(h)}{(\|S\|
_{L^{s}(v)})^{k}}.
$$
It is very simple to check that $R$ satisfies the required
properties. \end{proof}

\subsection{Two weight maximal function}

Our main new result is intimately related to a sharp two weight estimate for $M$. 

\begin{theorem} \label{twoweight}

Given a pair of weights  $u,\si$ and $p$, $1<p<\infty$, suppose that 
\begin{equation}\label{Orlicz-bump}
K=\sup_{Q} \left(
\frac{1}{ |Q| } \int_{Q} u(y)\, dy
\right)^{1/p}
\norm{  \si^{1/p'}  }_{X, Q}<\infty.
\end{equation}
where $X$ is a Banach function space such that its corresponding associate space $X'$ satisfies
$M_{X'}:L^p(\mathbb{R}^n)  \to L^p(\mathbb{R}^n)$. Then 
\begin{equation}\label{twoweightineq}
\Norm{M (f\si)}{L^p(u)}
  \lesssim   K\, \|M_{X'}\|_{\bdop(L^{p}(\R^n))}  \, \|f\|_{L^{p}(\si)}    
\end{equation}
In particular if $X=L_B$ with $B(t)=t^{p'} \,(1+\log^+ t)^{p'-1+\delta}$, $\delta>0$, then by \, \eqref{suff3}
$$    \|M_{X'}\|_{\bdop(L^{p}(\R^n))}  =  \|M_{\bar{B}}\|_{\bdop(L^{p}(\R^n))}  \approx  (p')^2 (\frac{1}{\delta})^{1/p}.$$
where the last $\approx$ is valid for $\delta \leq 1$.

\end{theorem}

This result together with some improvements  can be found in \cite{PR}.

\section{Dyadic theory}\label{dyadicSIO}

In this section we define an important class of dyadic model operators and recall a general result by which norm inequalities for maximal singular integral operators can be reduced to these dyadic operators. The result is due to Lerner  \cite{L2}, and comes from his approach to prove the $A_2$ theorem proved by the first author \cite{H}.  

We say that a dyadic grid, denoted $\cd$, is a collection of cubes in $\R^n$ with the following properties:\\
1) each $Q\in \cd$ satisfies $|Q|=2^{nk}$ for some $k\in \Z$; \\
2) if $Q,P\in \cd$ then $Q\cap P=\varnothing,P,$ or $Q$;\\
3)  for each $k\in \Z$, the family $\cd_k=\{Q\in \cd:|Q|=2^{nk}\}$ forms a partition of $\R^n$.

We say that a family of dyadic cubes $\cs \subset \cd$ is {\it sparse} if for each $Q\in \cs$, 
$$\Big|\bigcup_{\substack{Q'\in \cs\\ Q'\subsetneq Q}} Q'\Big|\leq \frac12|Q|.$$
Given a sparse family, $\cs$, if we define
$$E(Q):=Q\, \backslash \bigcup_{\substack{Q'\in \cs \\ Q'\subsetneq Q}} Q',$$
then \\
1) the family $\{E(Q)\}_{Q\in \cs}$ is pairwise disjoint\\
2) $E(Q)\subset Q$,\, and \\
3) $|Q|\leq 2|E(Q)|$.

If  $\cs \subset \cd$ is a sparse family we define the sparse Calder\'on-Zygmund operator associated to $\cs$ as 
$$T^{\cs}f:=\sum_{Q\in \cs}\avgint_Q f\,dx\cdot\chi_Q.$$

As already mentioned the key idea is to ``transplant'' the continuous case to the discrete version by means of the following theorem.


\begin{theorem} \label{dyadicboundCZ} Suppose that $X$ is a  quasi-Banach function space on $\R^n$ and $T$ is a Calder\'on-Zygmund operator. Then there exists a constant $c_T$
$$ \|T^*\|_{\bdop(X)}\leq c_T \,\sup_{\cs\subset \cd}\|T^{\cs}\|_{\bdop(X)}.$$
\end{theorem}

For Banach function spaces (without `quasi-'), this theorem is due to Lerner \cite{L2}. The stated generalization was obtained independently by Lerner and Nazarov \cite{LN} on the one hand, and by Conde-Alonso and Rey \cite{CAR} on the other hand. As a matter of fact, the last two papers only explicitly deal with the Calder\'on--Zygmund operator $T$ rather than the maximal truncation $T^*$, but the version above follows immediately from the same considerations, say, by combining \cite[Theorem 2.1]{HLP} and \cite[Theorem A]{CAR}.

We will not prove this theorem, we will simply mention that  a key tool is the  decomposition formula for functions 
found previously by Lerner \cite{L1} using the median. 
The main idea of this decomposition goes back to the work of Fujii \cite{Fujii2} where the standard average is used instead.

\section{Proof of Theorem \ref{MainResult}}

\subsection{Two lemmas}

Following the notion of dyadic singular integral operator mentioned in the section above we have the following key Lemma. 

\begin{lemma}\label{main assumption}
Let $w\in A_\infty$. Then for any sparse family $\cs \subset \cd$
\begin{equation}\label{sharpL1CF} 
\|T^{\cs}f \|_{L^1(w)}\leq
8[w]_{A_\infty}\|Mf\|_{L^1(w)}
\end{equation}
\end{lemma}

\begin{proof}
The left hand side equals for $f\geq 0$
\begin{equation*}
  \sum_{Q\in \cs}\avgint_Q f\,dx \,w(Q)
  \leq\sum_{Q\in\cs}\inf_{z\in Q} Mf(z)\,w(Q)
  \leq\sum_{Q\in\cs}\Big(\avgint_Q (Mf)^{1/2}dw\Big)^2\,w(Q).
\end{equation*}
By the Carleson embedding theorem, applied to $g=(Mf)^{1/2}$, we have
\begin{equation*}
  \sum_{Q\in\cs}\Big(\avgint_Q g\,dw\Big)^2\,w(Q)\leq 4K\|g\|_{L^2(w)}^2=4K\|Mf\|_{L^1(w)}
\end{equation*}
provided that the Carleson condition
\begin{equation}\label{eq:CarlesonToCheck}
  \sum_{\substack{Q\in\cs\\ Q\subseteq R}}w(Q)\leq K w(R)
\end{equation}
is satisfied. To prove \eqref{eq:CarlesonToCheck}, we observe that
\begin{equation*}
   \sum_{\substack{Q\in\cs\\ Q\subseteq R}}w(Q)
   =\sum_{\substack{Q\in\cs\\ Q\subseteq R}}\frac{w(Q)}{|Q|}|Q|
   \leq\sum_{\substack{Q\in\cs\\ Q\subseteq R}}\inf_{z\in Q}M(1_R w)(z)\cdot 2|E(Q)|
   \leq 2\int_R M(1_R w)(z)dz\leq 2[w]_{A_\infty}w(R).
\end{equation*}
This proves \eqref{eq:CarlesonToCheck} with $K=2[w]_{A_\infty}$, and the lemma follows.
\end{proof}

Actually, in the applications we have in mind we just need this for $w\in A_q\subset A_\infty$ for some fixed finite $q$. 

The second lemma is an extension of the well known Coifman-Rochberg Lemma: \\ 
$$
\mbox{If} \quad \g \in (0,1) \quad  \mbox{then }   \qquad M(\mu)^{\g} \in A_1 \quad  \mbox{with} \quad  [M(\mu)^{\g}]_{\strt{1.7ex}A_1}\le \frac{c_{n}}{1-\g}  \qquad 
$$

\begin{lemma} \label{Gene.Coifman-Rochberg}  
Let $A$ be a Young function and $u$ be a nonnegative  function such that  $M_{A} u(x)< \infty$ a.e. For $\g \in (0,1)$, there is a dimensional constant $c_{n}$ such that
\begin{equation}\label{cr}
[(M_{\strt{1.5ex}A}u)^{\g} ]_{A_1} \le c_{n}\,c_{\g}. 
\end{equation}
\end{lemma} 

A statement of this type is contained in \cite{CMP}, Proposition 5.32, but there it is suggested that the bound may also depend on the Young function $A$, while our version shows that it does not. This is again important for the quantitative consequences.

\begin{proof} 
We claim now that for each cube $Q$ and each $u$
\begin{equation} \label{eqn:A1wts-frac0}
 \avgint_Q M_{\strt{1.5ex} A}(u\chi_Q)(x)^\g\,dx
\leq
c_{n,\g}\,\|u\|_{\strt{1.5ex}A,Q}^\g.
\end{equation}
By homogeneity we may assume  $\|u\|_{\strt{1.5ex} A ,Q}=1$,  and so, in
particular, that\, $ \avgint_Q A(u(x))\,dx \leq 1.$

Now, the proof of  \eqref{eqn:A1wts-frac0} is based on the distributional estimate   \eqref{covering}. We split the integral at a level $\lambda\geq b_n$, yet to be chosen:
\begin{equation*}
\begin{split}
  \avgint_Q M_{A}(u\chi_Q)(x)^\g\,dx
  & = \frac{1}{|Q|}\int_0^\infty \g\,t^\g\, |\{ x\in Q :M_{A}(u\chi_Q)(x) > t\}|\, \frac{dt}{t} \\
  & \leq \frac{1}{|Q|}\int_0^{\lambda}\gamma t^\gamma|Q|\frac{dt}{t}
    +\frac{1}{|Q|}\int_\lambda^\infty \gamma t^\gamma a_n\int_Q A\Big(b_n\frac{|u(x)|}{t}\Big)dx\frac{dt}{t} \\
  &\leq \lambda^\gamma+\frac{1}{|Q|}\int_\lambda^\infty\gamma t^{\gamma}a_n\int_Q\frac{b_n}{t} A(|u(x)|)\ud x\frac{dt}{t} \\
  &\leq\lambda^\gamma+a_nb_n\gamma\int_\lambda^\infty t^{\gamma-2}dt
  =\lambda^\gamma+a_nb_n\frac{\gamma}{1-\gamma}\lambda^{\gamma-1}.
\end{split}
\end{equation*}
With $\lambda=a_nb_n$, we arrive at
\begin{equation*}
  \avgint_Q M_{A}(u\chi_Q)(x)^\g\,dx\leq\frac{(a_nb_n)^\gamma}{1-\gamma},
\end{equation*}
which is \eqref{eqn:A1wts-frac0}, in view of our normalization that $\|u\|_{A,Q}=1$.

We will use the following fact  that can be also found in \cite{CMP}:  for every $Q$ 
\begin{equation} \label{eqn:local-max2}
M_{\strt{1.5ex} A }(u\,\chi_{\re^n\setminus3\,Q})(x) \approx \sup_{P\supset Q}  \|u\,\chi_{\re^n\setminus3\,Q}\|_{\strt{1.5ex} A ,P}  \qquad x\in Q
\end{equation}
where the constant in the direction $\leq$ is dimensional (actually $3^n$). \eqref{eqn:local-max2}
shows that $M_{\strt{1.5ex} A }(f\,\chi_{\re^n\setminus3Q})$ is essentially
constant on $Q$.

Finally since $A$ is a Young, the triangle inequality combined with \eqref{eqn:A1wts-frac0}  and \eqref{eqn:local-max2} 
gives for every $y\in Q$,
\begin{align*}
&  \avgint_Q M_{\strt{1.5ex}A}u(x)^\g\,dx  \\
&\qquad \leq
3^n\avgint_{3Q} M_{\strt{1.5ex}A}(u\chi_{3Q})(x)^\g\,dx +
 \avgint_Q M_{\strt{1.5ex}A}(u\chi_{\re^n\setminus 3Q})(x)^\g\,dx.
\\
&\qquad \leq
c_{n,\g}\, \|u\|_{\strt{1.5ex}A,3Q}^\g +
3^n\,\big(\sup_{P\supset Q}\,\|u\,\chi_{\re^n\setminus3Q}\|_{A ,P}\big)^\g \\
& \qquad \leq c_{n,\g}\,M_{\strt{1.5ex}A}u(y)^\g.
\end{align*}
This completes the proof of the lemma.
\end{proof}

\subsection{Proof of Theorem \ref{MainResult}}

We have to prove 
$$
\|T^*f \|_{ L^{p}(w) }\,   \leq  c_T \,p'\,  \|M_{\bar{A}}\|_{\bdop(L^{p'}(\R^n))}\,  \|f \|_{ L^{p}(M_{A}(w^{1/p})^{p})  }  \qquad w\geq 0. 
$$
and if we use the notation $A_p(t)=A(t^{1/p})$ this becomes 
$$
\|T^*f \|_{ L^{p}(w) }\,   \leq     c_T \,p'\,  \|M_{\bar{A}}\|_{\bdop(L^{p'}(\R^n))}\,  \|f \|_{ L^{p}(M_{A_p}(w))   }.     
$$
By Theorem \ref{dyadicboundCZ}  everything is reduced to proving that 
\begin{equation}\label{LpBddyadicCZ}
\|T^{\cs}f \|_{ L^{p}(w) }\,   \lesssim     \,p'\,  \|M_{\bar{A}}\|_{\bdop(L^{p'}(\R^n))}\,  \|f \|_{ L^{p}(M_{A_p}(w))   } \qquad \cs \subset \cd.
\end{equation}
Now, by duality we will prove the equivalent estimate 
\begin{equation*}
\|T^{\cs}(fw) \|_{ L^{p'}(M_{A_p}(w)^{1-p'}) }\,   \lesssim      \,p'\,  \|M_{\bar{A}}\|_{\bdop(L^{p'}(\R^n))}\,  \|f \|_{L^{p'}(w)}.  
\end{equation*}
because the adjoint of $T^{\cs}$ (with respect to the Lebesgue measure) is itself.  

The main claim is the following: 

\begin{lemma}\label{lem:mainClaim}
\begin{equation}\label{claim1}
\|T^{\cs}(g) \|_{ L^{p'}(M_{A_p}(w)^{1-p'}) } \lesssim  p'\,\|M(g) \|_{ L^{p'}(M_{A_p}(w)^{1-p'}) } \qquad \cs \subset \cd  \quad g\geq  0.
\end{equation}
\end{lemma}

\begin{proof}
Now 
$$
\|T^{\cs}(g) \|_{ L^{p'}(M_{A_p}(w)^{1-p'}) } =\left\|\frac{T^{\cs}(g)}{M_{A_p}w}\right\|_{ \strt{1.7ex} L^{p'}( M_{A_p}w) }$$
and by duality we have that for some nonnegative $h$ with $\|h\|_{L^{p}(M_{A_p}w)}=1$
$$
\left\|\frac{T^{\cs}(g)}{M_{A_p}w}\right\|_{ \strt{1.7ex} L^{p'}( M_{A_p}w) }
=\int_{\mathbb{R}^{n}}T^{\cs}(g)\,h\,dx 
$$
Now, by Lemma \ref{smallLemma} with \, $s=p$ \, and \, $v=M_{A_p}w$\, there
exists an operator~$R$ such that\\
(A) $h\le R(h)$\\
(B) $\|R(h)\|_{L^{p}(M_{A_p}w)}\le 2\|h\|_{L^{p}(M_{A_p}w)}$\\
(C) $[R(h)(M_{A_p}w)^{1/p}]_{A_1}\le cp'.$

Hence, 
$$
\|T^{\cs}(g) \|_{ L^{p'}(M_{A_p}(w)^{1-p'}) } \leq \int_{\mathbb{R}^{n}}T^{\cs}(g)\,Rh\,dx. 
$$
Next we plan to replace $T^{\cs}$ by $M$  by using Lemma \ref{main assumption}. To do this we to estimate 
the $A_q$ constant of $Rh$, for a fixed $q>1$ (in fact, $q=3$) using property (C) combining the following two facts. The first one is well known, is the easy part 
of the factorization theorem, if \, $w_1,w_2\in A_1$,\, then \,$w=w_1w_2^{1-p}\in A_p$, \,and 
$$[w]_{A_p}\le [w_1]_{A_1}[w_2]_{A_1}^{p-1}
$$
The second fact is Lemma~\ref{Gene.Coifman-Rochberg}

Now if we choose $\g=\frac12$ in Lemma \ref{Gene.Coifman-Rochberg},
\begin{eqnarray*}
[R(h)]_{A_\infty}\ \lesssim\ [R(h)]_{\strt{1.7ex}A_3}&=&[R(h)(M_{A_p}w)^{\frac1p}\big((M_{A_p}w)^{\frac{1}{2p}}\big)^{1-3}]_{\strt{1.7ex}A_3}\\
&\leq& [R(h)(M_{A_p}w)^{\frac1p}]_{A_1}[(M_{A_p}w)^{\frac{1}{2p}}]_{\strt{1.7ex}A_1}^{3-1}\\
&\leq & c_n\,p'\, [ M_{A}(w^{1/p})  ^{\frac12}]_{\strt{1.7ex}A_1}^{3-1}\\
&\le& c_np'
\end{eqnarray*}
by the lemma and since \,$A_p(t)=A(t^{1/p})$.

Therefore, by Lemma \ref{main assumption} and by
properties (A) and (B) together with H\"older,
\begin{eqnarray*}
  \int_{\mathbb{R}^{n}}T^{\cs}(g)h\,dx&\le&
  \int_{\mathbb{R}^{n}}T^{\cs}(g)R(h)\,dx
  \lesssim\,[R(h)]_{A_\infty}\int_{\mathbb{R}^{n}}M(g)R(h)\,dx\\
  &\lesssim&
  \,p'\,\left\|\frac{M(g)}{M_{A_p}w}\right\|_{L^{p'}(M_{A_p}w)}\|Rh\|_{L^{p}(M_{A_p}w)}=
  c_{N}\,p'\,\left\|\frac{M(g)}{M_{A_p}w}\right\|_{L^{p'}(M_{A_p}w)}.
\end{eqnarray*}
This proves claim \eqref{claim1}.
\end{proof}

With \eqref{claim1}, the proof of Theorem~\ref{MainResult} is reduced to showing that
\begin{equation*}
\|M(fw) \|_{ L^{p'}(M_{A_p}(w)^{1-p'}) }\,   \leq c \|M_{\bar{A}}\|_{\bdop(L^{p'}(\R^n))}\,  \|f \|_{L^{p'}(w)}
\end{equation*} 
for which we can apply the two weight theorem for the maximal function (Theorem \ref{twoweight}) to the couple of weights 
\,$(M_{A_p}(w)^{1-p'}, w)$\, with exponent $p'$. We need then to compute \eqref{Orlicz-bump}: (We reproduce this short calculation from \cite{CMP}, Theorem  6.4, for completeness.)
 $$
 \left(
\frac{1}{ |Q| } \int_{Q}  M_{A_p}(w)^{1-p'}\, dy
\right)^{1/p'}
\norm{  w^{1/p}  }_{A, Q} \leq   \|w\|_{A_p,Q} ^{-1/p}\,\norm{  w^{1/p}  }_{A, Q} = \|w^{1/p}\|_{A,Q} ^{-1}\, \norm{  w^{1/p}  }_{A, Q} =1,
$$
since $A_p(t)=A(t^{1/p})$.  Hence
$$
\|M(fw) \|_{ L^{p'}(M_{A}(w)^{1-p'}) } \leq c\, \|M_{\bar{A}}\|_{\bdop(L^{p'}(\R^n))}\,\|f \|_{L^{p'}(w)}
$$
concluding the proof of the theorem.

\section{Proof of Theorem \ref{EndPointResult}}

To prove the Theorem we follow the basic scheme as in \cite{P2} (see also \cite{LOP}, \cite{HP}).

Thanks to Theorem~\ref{dyadicboundCZ}, it is enough to prove the following dyadic version:
\begin{prop} 
Let $\cd$ be a dyadic grid and let $\cs \subset \cd$ be a sparse family. Then, there is a universal constant $c$ independent of $\cd$ and $\cs$ such that 
for any \,$0<\epsilon \leq 1$
\begin{equation} \label{EndpointDyadicT}
 \|T^{\cs}f\|_{L^{1,\infty}(w)} \leq  \frac{c}{\epsilon}
 \int_{\R^n} |f(x)|\,M_{L(\log L)^{\epsilon}} (w)(x)\,dx   \qquad w\geq 0
\end{equation}
\end{prop}

Note that in order to deduce Theorem~\ref{EndPointResult} from the Proposition above, we need the full strength of Theorem~\ref{dyadicboundCZ} with quasi-Banach function space, because the space $L^{1,\infty}$  is not normable. It is also possible to prove Theorem~\ref{EndPointResult} directly (without going through the dyadic model); this was our original approach, since the quasi-Banach version of Theorem~\ref{dyadicboundCZ} was not yet available at that point.  However, we now present a proof via the dyadic model, which simplifies the argument.

Recall that the sparse Calder\'on-Zygmund operator\, $T^{\cs}$\, is defined by,
$$T^{\cs}f=\sum_{Q\in \cs}\,\avgint_Q f\,dx\cdot\chi_Q.$$
By homogeneity on $f$ it would be enough to prove 
\begin{equation*}
w\{x\in \R^n: T^{\cs}f(x)>2\}  \leq  \frac{c}{\epsilon}
 \int_{\R^n} |f(x)|\,M_{L(\log L)^{\epsilon}} (w)(x)\,dx.   
\end{equation*}

We consider the the CZ decomposition of $f$ with respect to the grid $\cd$ at level $\la=1$. There is  family of pairwise disjoint cubes $\{Q_j\}$ from $\cd$ such that
$$1<  \frac{1}{|Q_j|} \int_{Q_j} |f|\le 2^n $$
Let \, $\Omega=\bigcup_jQ_j$ \, and \, $\widetilde \Omega=\bigcup_j3Q_j$ \,. The
``good part'' is defined by
$$g=\sum_jf_{Q_j}\chi_{Q_j}(x)+f(x)\chi_{\Omega^c}(x),$$
and it satisfies $\Norm{g}{L^\infty}\leq 2^n$ by construction.
The ``bad part'' \,$b$\, is \,$b=\sum\limits_{j}b_{j}$ where \,$b_{j}(x)=(f(x)-f_{Q_{j}})\chi_{\strt{1.7ex}Q_{j}}(x)$. Then, $f=g+b$ and we split the level set as 
\begin{eqnarray*}
w\{x\in {\mathbb R}^d: T^{\cs}f(x)>2\}&\le& w(\widetilde \Omega)+w\{x\in
(\widetilde \Omega)^c: T^{\cs}b(x)>1\}\\
&+&w\{x\in (\widetilde \Omega)^c:T^{\cs}g(x)>1\}=I+II+III.
\end{eqnarray*}
%


As in  \cite{P2}, the most singular term is $III$. We first deal with the easier terms $I$ and $II$, which actually satisfy the better bound
\begin{equation*}
  I+II \leq c_{T}\, \|f\|_{L^1(Mw)}. 
\end{equation*}
The first is simply the classical Fefferman-Stein inequality \eqref{fs}.


To estimate $II=w\{x\in (\widetilde \Om)^c: |T^{\cs}b(x)|>1\}$\, we argue as follows:
$$
w\{x\in (\widetilde \Om)^c: |T^{\cs}b(x)|>1\} \leq  \int_{ \R^{n} \setminus \tilde{\Omega} }
|T^{\cs}b(x)|\, w(x)dx \lesssim  \sum_{j}   \int_{ \R^{n} \setminus \tilde{\Omega} }
|T^{\cs}(b_{j})(x)|\, w(x)dx
$$
$$
\lesssim  \sum_{j}   \int_{ \R^{n} \setminus 3Q_{j}}
|T^{\cs}(b_{j})(x)|\, w(x)dx
$$
We fix one of these $j$ and  estimate now $T^{\cs}(b_{j})(x)$ for \,$x\notin 3Q_j$:
$$
T^{\cs}(b_{j})(x) =  \sum_{Q\in \cs}\, \avgint_Q b_{j} \,dy\cdot\chi_Q(x) = \sum_{Q\in \cs, Q\subset Q_j}+ \sum_{Q\in \cs , Q\supset Q_j} =  \sum_{Q\in \cs , Q\supset Q_j} 
$$
since \,$x\notin Q_j$. Now, this expression is equal to   
$$
\sum_{Q\in \cs , Q\supset Q_j}   \frac{1}{|Q|} \int_{Q_j} (f(y)-f_{Q_j}) \,dy\, \cdot\chi_Q(x)  
$$
and this expresion is zero by the key cancellation: \, $\int_{Q_j} (f(y)-f_{Q_j}) \,dy=0$. Hence $II=0$, and we are only left with the singular term $III$.

\subsection{Estimate for part $III$}

We now consider the last term $III$, the singular part. We apply Chebyschev's inequality and then \eqref{LpBddyadicCZ} with exponent $p$ and functional $A$, that will be chosen soon:
\begin{eqnarray*}  
&III &= w\{x\in (\widetilde \Om)^c: T^{\cs}g(x)>1\}\\
&&\le \Norm{T^{\cs}g}{L^p(w\chi_{(\widetilde \Om)^c})}^p \\
&&  \lesssim     \,(p')^p\,  \|M_{\bar{A}}\|^p_{\bdop(L^{p'}(\R^n))} \,\int_{{\mathbb R}^n}|g|^pM_{A_p}(w\chi_{(\widetilde \Om)^c})dx\\
&&\lesssim   \,(p')^p\,  \|M_{\bar{A}}\|^p_{\bdop(L^{p'}(\R^n))}\,\int_{{\mathbb R}^n}|g|\, M_{A_p}(w\chi_{(\widetilde\Om)^c})dx,
\end{eqnarray*}
using the boundedness of $g$ by $2^n\lesssim 1$, and denoting $A_p(t)=A(t^{1/p})$.

Now,  we will make use of \eqref{eqn:local-max2} again: for an arbitrary Young
function $B$, a nonnegative function $w$ with $M_{B} w(x)< \infty $
a.e., and a cube $Q$, we have
\begin{equation}
M_{B}( \chi_{_{ \R^{n} \setminus 3Q }} w )(y)
\approx M_{B}( \chi_{_{ \R^{n} \setminus 3Q }} w )(z)
\label{geom}
\end{equation}
for each $y,z \in Q$ with dimensional constants.  Hence, combining \eqref{geom} with the definition of $g$ we have 
\begin{equation*}
\begin{split}
  \int_{\Om} |g|M_{A_p}(w\chi_{(\widetilde \Om)^c})dx
  &\lesssim\sum_{j}\int_{ Q_{j} } \absval{f(x)}\,dx  \inf_{Q_{j}}
  M_{A_p}(w\chi_{(\widetilde\Om)^c}) \\
  &\lesssim   \int_{ \Om } \abs{f(x)}\,  M_{A_p}w(x)\ud x,
\end{split}
\end{equation*}
and of course
\begin{equation*}
\int_{\Om^c}|g|M_{A_p}(w\chi_{(\tilde\Om)^c})\ud x
\leq\int_{\Om^c}|f|M_{A_p} w\ud x.
\end{equation*}
Combining these, we have 
$$III \lesssim   \,(p')^p\,  \|M_{\bar{A}}\|^p_{\bdop(L^{p'}(\R^n))}\,\int_{{\mathbb R}^d}|f|\,M_{A_p}(w)dx.$$

We optimize this estimate by choosing an appropriate $A$. 
To do this we apply now Lemma  \ref{suff2} and more particularly to the example considered in \eqref{suff3}, 
 namely $B$ is so that $\bar{B}(t)= A(t)=  t^{p}(1+\log^{+} t)^{p-1+\delta},$ $\delta>0$.  Then 
$$
\|M_{\bar{A}}\|_{\bdop(L^{p'}(\R^n))}
\leq c_{n}\,
\left(
\int_{  1}^{\infty} \left( \frac{ t }{ A(t)} \right)^{p'}
A'(t)\, dt \right)^{1/p'} 
\lesssim p \left( \frac{1}{\delta}\right)^{1/p'}  \qquad 0<\delta \leq 1
$$
Then $A_p(t)=A(t^{1/p}) \leq  t(1+\log^{+} t)^{p-1+\delta}$ and we have 
$$
III \lesssim   (p')^p \, \left( \frac{1}{\delta}\right)^{p-1} \, \int_{{\mathbb R}^d}|f|\, M_{L(\log L)^{p-1+\delta}} (w)(x)\,dx.
$$
Now if we choose $p$ such that
$$
p-1=\frac{\epsilon}{2}=\delta <1
$$
then \, $ (p')^p (\frac{1}{\delta})^{p-1} \lesssim \frac{1}{\epsilon}$\, if\, $\epsilon<1$.

This concludes the proof of \eqref{EndpointDyadicT}, and hence of Theorem~\ref{EndPointResult}.

\section{Proof of Corollary \ref{Lloglleftbumpcondit}} \label{proofMainResult}

We follow very closely the argument given in \cite{CP1}, the essential difference is that we compute in a more precise way the constants involved. We consider the set  
$$\Omega=\{ x\in\R^n : T^*(f\si)(x)>1 \}$$
Then by homogeneity it is enough to prove 
\begin{equation}
u(\Omega)^{1/p}  \lesssim  \frac{1}{\delta}\,K\, (\frac{1}{\delta})^{1/p'}   \|f\|_{L^{p}(\si)}    
\end{equation}
where we recall that 
\begin{equation}\label{bumpleftcond}
K=\sup_Q\, \|u^{1/p}\|_{ \strt{1.7ex} L^p(\log L)^{p-1+\delta},Q}
\left(\frac{1}{|Q|}\int_Q \si \,dx\right)^{1/p'} < \infty
\end{equation}

Now, by
duality, there exists a non-negative function $h\in L^{p'}(\R^{n})$,
 $\|h\|_{L^{p'}(\R^{n})}=1$, such that
$$
u(\Omega)^{1/p}
=  \|u^{1/p}\chi_{\Omega}\|_{L^{p}(\R^{n})} =  \int_{\Omega}u^{1/p}h\,dx= u^{1/p}h(\Omega)
\lesssim  \frac{1}{\e}\int_{\R^n} |f|M_{L(\log L)^{\e}}(u^{1/p}h)\,\si dx
$$
$$
\leq  \frac{1}{\e} \left(\int_{\R^n} |f|^p\,\si dx\right)^{1/p}
\left(\int_{\R^n} M_{L(\log L)^{\e}}(u^{1/p}h)^{p'}\, \si dx\right)^{1/p'},
$$
where we have used inequality  \eqref{goalT*} from Theorem \ref{EndPointResult} and then H\"older's inequality.  Therefore everything is reduced to understanding a two weight estimate for $M_{L(\log L)^{\e}}$.

We need the following Lemma that can be found in \cite{P1} or in \cite{CMP} Appendix A, Proposition A.1

\begin{lemma}\label{decomp}  Given a Young function $A$, suppose $f$ is a non-negative
function such that $\|f\|_{A,Q}$ tends to zero as $l(Q)$ tends to
infinity.  Given $a>2^{n+1}$, for each $k\in \Z$ there exists a disjoint
collection of maximal dyadic cubes $\{Q_j^k\}$ such that for each $j$,
\begin{equation}
 a^k < \|f\|_{A,Q_j^k} \leq 2^na^k,
\label{decomp0}
\end{equation}
and
\[ \{ x \in \R^n : M_Af(x) >  4^na^k \} \subset \bigcup_j 3Q_j^k. \]
Further, let $D_k= \bigcup_j Q_j^k$ and
$E_j^k =  Q_j^k\setminus(Q_j^k\cap D_{k+1})$.  Then the $E_j^k$'s are
pairwise disjoint for all $j$ and $k$  and there exists a constant
 $\alpha>1$, depending only on $a$, such that
 $|Q_j^k|\leq \alpha|E_j^k|$.
\end{lemma}

Fix a function $h$  bounded with compact support. Fix $a>2^{n+1}$;  for $k\in\Z$ let
\[ \Omega_k = \{ x\in \R^n : 4^na^k < M_Af(x) \leq 4^na^{k+1} \}. \]
Then by Lemma \ref{decomp},
\[ \Omega_k \subset \bigcup_j 3Q_j^k,
\quad \text{ where } \quad
 \|f\|_{A,Q_j^k} > a^k. \]

We will use a  generalization of H\"older's inequality
due to O'Neil \cite{O1}.  (Also see Rao and Ren
\cite[p.\ 64]{RR}.) We include a proof for the reader's convenience. 

\begin{lemma}\label{lem:YoungHolder}
Let $A,B$ and $C$ be Young functions such that
\begin{equation}\label{eq:ABCinvAss}
  B^{-1}(t)C^{-1}(t)\leq\kappa A^{-1}(t),\quad t>0.
\end{equation}
Then for all functions $f$ and $g$ and all cubes $Q$,
\begin{equation}\label{holder}
  \|fg\|_{A,Q} \leq 2\kappa\|f\|_{B,Q}\|g\|_{C,Q}. 
\end{equation}
\end{lemma}

\begin{proof}
The assumption \eqref{eq:ABCinvAss} says that if $A(x)=B(y)=C(z)$, then $yz\leq \kappa x$. Let us derive a more applicable consequence:

Let $y,z\in[0,\infty)$, and assume without loss of generality (by symmetry) that $B(y)\leq C(z)$. Since Young functions are onto, we can find a $y'\geq y$ and $x\in[0,\infty)$ such that $B(y')=C(z)=A(x)$. Then \eqref{eq:ABCinvAss} tells us that $yz\leq y'z\leq\kappa x$. Since $A$ is increasing, it follows that
\begin{equation}\label{eq:ABCdirect}
  A\Big(\frac{yz}{\kappa}\Big)\leq A(x)=C(z)=\max(B(y),C(z))\leq B(y)+C(z).
\end{equation}

Let then $s>\Norm{f}{B}$ and $t>\Norm{g}{C}$. Then, using \eqref{eq:ABCdirect},
\begin{equation*}
  \avgint_Q A\Big(\frac{\abs{fg}}{\kappa st}\Big)
  \leq\avgint_Q B\Big(\frac{\abs{f}}{s}\Big)+\avgint_Q C\Big(\frac{\abs{g}}{t}\Big)\leq 1+1,
\end{equation*}
and hence
\begin{equation*}
  \avgint_Q A\Big(\frac{\abs{fg}}{2\kappa st}\Big)\leq\frac{1}{2}\avgint_Q A\Big(\frac{\abs{fg}}{\kappa st}\Big)\leq 1.
\end{equation*}
This proves that $\Norm{fg}{A}\leq 2\kappa st$, and taking the infimum over admissible $s$ and $t$ proves the claim.
\end{proof}

If $A(t)=t(1+\log^{+}t)^{\e}$,  the goal is to ``break'' $M_{A}$ in an optimal way, with functions $B$ and $C$ so that one of them, for instance $B$, has to be \,$B(t)=t^{p}(1+\log^{+} t)^{p-1+\delta}$\, coming from \eqref{bumpleftcond}.

We can therefore estimate $M_{A}$ using  Lemma \ref{decomp} as follows:
\begin{align*}
\int_{\R^n} (M_A(u^{1/p}\,h))^{p'}\si\,dx
= & \sum_k \int_{\Omega_k} (M_A(u^{1/p}\,h))^{p'}\si\,dx \\
\leq & c \sum_k a^{kp'}\si(\Omega_k) \\
\leq & c \sum_{j,k} a^{kp'}\si(3Q_j^k) \\
\leq & c \sum_{j,k} \si(3Q_j^k)\|u^{1/p}\,h\|_{A,Q_j^k}^{p'}. \\
\leq & c\sum_{j,k} \si(3Q_j^k)\|u^{1/p}\|_{B,Q_j^k}^{p'}
\|h\|_{C,Q_j^k}^{p'}, \\
\intertext{by \eqref{holder}.   Now since  $\|u^{1/p}\|_{B,Q_j^k} \leq 3^n\|u^{1/p}\|_{B,3Q_j^k}$,  we can apply condition \eqref{bumpleftcond}, and since the $E_j^k$'s are disjoint,}
\leq & c\sum_{j,k} \left(\frac{1}{|3Q_j^k|}\int_{3Q_j^k}\si\,dx\right)
\|u^{1/p}\|_{B,3Q_j^k}^{p'}\|h\|_{C,Q_j^k}^{p'}|E_j^k| \\
\leq & K^{p'}\,\sum_{j,k} \int_{E_j^k}M_C(h)^{p'}\,dx \\
\leq & K^{p'}\, \int_{\R^n}  M_C(h)^{p'}\,dx. \\
\leq & K^{p'}\,  \|M_{C}\|^{p'}_{\bdop(L^{p'}(\R^n))}\,\int_{\R^n} h^{p'}\,dx.
\end{align*}
If we choose $C$ such that $M_{C}$ is bounded on $L^{p'}(\R^{n})$, namely it must satisfy the tail condition \eqref{Bp}. We are left with choosing the appropriate \,$C$\,. Now, $1<p<\infty$ and $\delta>0$ are fixed from condition  \eqref{bumpleftcond} but $\e>0$ is free and will be chosen appropriately close to $0$. To be more precise we need to choose  $0<\e <\delta/p$ and let $\eta= \delta-p\e $.  Then
\begin{align*}
 A^{-1}(t)
\approx & \frac{t}{(1+\log^+t)^\e} \\
= & \frac{t^{1/p}}{(1+\log^+t)^{\e+(p-1+\eta)/p}  }
\times t^{1/p'} (1+\log^+t)^{(p-1+\eta)/p} \\
= & B^{-1}(t)C^{-1}(t),
\end{align*}
where
\[ B(t) \approx t^p  (1+\log^+t)^{(1+\e)p-1+\eta}=
t^p(1+\log^+t)^{p-1+\delta}\]
and
\[  C(t) \approx t^{p'}(1+ \log^+t)^{-1-(p'-1)\eta}.\]
These manipulations follow essentially  O'Neil \cite{O2} but we need to be careful with the constants.

It follows at once from Lemma~\ref{LemmaBp} that
$$  \|M_{C}\|_{\bdop(L^{p'}(\R^n))} \lesssim \Big(\frac{1}{\eta}\Big)^{1/p'}=(\frac{1}{\delta-p\e})^{1/p'},$$
where we suppress the multiplicative dependence on $p$.
Finally if we choose $\e=\frac{\delta}{2p}$ we get the desired result:
\begin{equation}
u(\Omega)^{1/p}  \lesssim  \frac{1}{\delta}\,K\, (\frac{1}{\delta})^{1/p'}   \|f\|_{L^{p}(\si)}    
\end{equation}
This completes the proof of part (a) of Corollary \ref{proofMainResult}. 

To prove part (b) we combine Lerner's theorem  \ref{dyadicboundCZ},
$$ \|T^*f\|_{L^p(u)}\leq c_T \,\sup_{\cs\subset \cd} \|T^{\cs}f\|_{L^p(u)},$$
with the characterization of the two-weight inequalities for $T^{\cs}$ from \cite{LSU} by testing conditions: a combination of their characterizations for weak and strong norm inequalities shows in particular that
$$
\|T^{\cs}(.\si)\|_{L^p(\si) \to L^{p}(u) }  \eqsim \|T^{\cs}(.\si)\|_{L^p(\si) \to L^{p,\infty}(u) }  +  
\|T^{\cs}(.u)\|_{L^{p'}(u) \to L^{p',\infty}(\si) }
$$

Now, as it is mentioned after the statement of Corollary \ref{Lloglleftbumpcondit}, since $T^{\cs}$  satisfies estimate \eqref{goalT*} (see  \eqref{EndpointDyadicT}) 
we can apply the same argument as the just given to both summands and since that estimate has to be independent of the grid and we must take the two weight constant $K$ 
over all cubes, not just for those from the specific grid. This concludes the proof of the corollary.

\section{Conjectures} \label{conjetura}

A conjecture related to Corollary~\ref{Lloglleftbumpcondit} is as follows:

\begin{conjecture}

Let $T^*$, $p,u,\si$ as above.  Let  $X$ is a Banach function space so that its corresponding associate space $X'$ satisfies
$M_{X'}:L^{p'}(\mathbb{R}^n)  \to L^{p'}(\mathbb{R}^n)$. If  
\begin{equation}
K=\sup_Q\, \|u^{1/p}\|_{ \strt{1.7ex} X,Q}
\left(\frac{1}{|Q|}\int_Q \si\,dx\right)^{1/p'} < \infty,
\label{Orlcizleftftbumpcondit}
\end{equation}
then 
\begin{equation}
\|T^*(f\si)\|_{L^{p,\infty}(u)}    \lesssim\,K\,  \|M_{X'} \,\|_{\mathcal{B}(L^{p'}(\R^n))}  \|f\|_{L^{p}(\si)}.    
 \label{conjecture(p,p)}
\end{equation}
As a consequence, if $Y$ is another Banach function space with $M_{Y'}:L^{p}(\mathbb{R}^n)  \to L^{p}(\mathbb{R}^n)$ and if 
\begin{equation}
K=\sup_Q\, \|u^{1/p}\|_{ \strt{1.7ex} X,Q}
\left(\frac{1}{|Q|}\int_Q \si\,dx\right)^{1/p'} + 
\left( \frac{1}{|Q|}\int_Q u \,dx \right)^{1/p}   \, 
 \|\sigma^{1/p'}\|_{ \strt{1.7ex} Y,Q}
  < \infty,
\end{equation}
then 
\begin{equation}
\|T^*(f\si)\|_{L^{p}(u)}    \lesssim\,K\,  \left( \|M_{X'} \,\|_{\mathcal{B}(L^{p'}(\R^n))} + \|M_{Y'} \,\|_{\mathcal{B}(L^{p}(\R^n))}   \right)\,\|f\|_{L^{p}(\si)}    
\end{equation}

\end{conjecture}

This is a generalization of the conjecture stated in \cite{CRV} which arises from the work \cite{CP1,CP2}. We also refer to the recent papers \cite{L,TV} for further results in this direction.

If we could prove this, we would get as corollary:

\begin{cor}
\begin{equation}
 \|T^{*}\|_{\mathcal{B}(L^p(w))}
\leq c[w]_{A_p}^{1/p}\big( [w]_{A_\infty}^{1/p'}+[\si]_{A_\infty}^{1/p} \big) 
\end{equation}
\end{cor}

This last result itself is known \cite{HL}  (see also \cite{HLP} for a more general case), but not as a corollary of a general two-weight norm inequality.

\vspace{3 mm}

\noindent
Department of Mathematics and Statistics\\
P.O. Box 68 (Gustaf H\"allstr\"omin katu 2b)\\
FI-00014 University of Helsinki, Finland

\vspace{1 mm}

\noindent \emph{E-mail address}: {\tt tuomas.hytonen@helsinki.fi}

\vspace{3.5 mm}

\noindent
Department of Mathematics\\ 
University of the Basque Country UPV/EHU, Leioa, Spain, and\\
IKERBASQUE, Basque Foundation for Science, Bilbao, Spain
 
\vspace{1 mm}

\noindent \emph{E-mail address}: {\tt c.perez@ikerbasque.org}

\end{document}